\documentclass[12pt,amstex]{amsart}

\usepackage{mathptmx}
\usepackage{mathrsfs}

\usepackage{verbatim}
\usepackage{url}
\usepackage[all]{xy}
\usepackage{color}

\usepackage[colorlinks=true,citecolor=blue]{hyperref}

\usepackage{stmaryrd}
\usepackage{epsfig}
\usepackage{amsmath}
\usepackage{amssymb}

\usepackage{amscd}
\usepackage{graphicx}
\usepackage{pstricks}

\theoremstyle{plain}
\newtheorem{thm}{Theorem}[section]

\newtheorem{prop}[thm]{Proposition}
\newtheorem{ques}[thm]{Question}
\newtheorem{cor}[thm]{Corollary}
\theoremstyle{definition}

\theoremstyle{remark}
\newtheorem{rem}[thm]{Remark}

\def \N {\mathbb N}

\def \Z {\mathbb Z}

\def \a {\alpha }
\def \b {\beta}
\def \ep {\epsilon}

\def \D {\Delta}

\begin{document}
\title{A non-PI minimal system is Li-Yorke sensitive}

\author{Song Shao}
\author{Xiangdong Ye}

\address{Wu Wen-Tsun Key Laboratory of Mathematics, USTC, Chinese Academy of Sciences and
Department of Mathematics, University of Science and Technology of China,
Hefei, Anhui, 230026, P.R. China.}

\email{songshao@ustc.edu.cn}
\email{yexd@ustc.edu.cn}


\thanks{This research is supported by NNSF of China (11371339, 11431012, 11571335) and by ¡°the Fundamental Research Funds for the Central Universities¡±.}

\date{}

\begin{abstract}
It is shown that any non-PI minimal system is Li-Yorke sensitive. Consequently, any minimal system with nontrivial
weakly mixing factor (such a system is non-PI) is Li-Yorke sensitive, which answers affirmatively an open question by Akin and Kolyada in \cite{AK03}.
\end{abstract}

\maketitle





\section{Introduction}

\subsection{Topological dynamics}\
\medskip

First we recall some basic notions in topological dynamics.
{\it A topological dynamical system} $(X,T)$ is a compact metric space
$(X,\rho)$ endowed with a continuous surjective map $T: X\rightarrow X$.
For simplicity, we only consider homeomorphisms in this paper.

\medskip

Recall that $(X, T )$ is {\em transitive} if for each pair of opene (i.e. non-empty and
open) subsets $U$ and $V$, $N(U,V ) = \{n \in \Z : U \cap T^{-n}V\neq \emptyset\}$ is non-empty.
$(X, T)$ is (topologically) {\em weakly mixing} if $(X\times X, T\times T)$ is transitive.
A point $x\in X$ is a {\em transitive point} if its orbit $Orb (x,T)=\{T^nx: n\in \Z\}$ is
dense in $X$. If every point of $X$ is transitive, we say $(X,T)$ is {\em minimal}. A point is minimal if its orbit closure is a minimal subsystem.

\medskip

A pair $(x,y)\in X\times X$ is said to be {\em proximal} if $\liminf_{n\to \infty} \rho(T^nx,T^ny)=0$, and it is called {\em asymptotic} when $\lim_{n \to \infty} \rho(T^nx,T^ny)=0$. The set of proximal pairs is denoted by $P(X,T)$ or $P$ when the system is clear. $P$ is a reflexive, symmetric, $T$-invariant relation but in general is not transitive or closed. For $x\in X$, the set $P[x]=\{y\in X: (x,y)\in P\}$ is called the {\em proximal cell} of $x$.
An important result concerning the proximal cell is that for any $x\in X$, $P[x]$ contains a minimal point; more precisely,
every minimal subset of $\overline{Orb(x, T)}$ meets $P[x]$ \cite[Theorem 5.3.]{Au88}. It follows immediately that
if $P[x]$ is a singleton, then $x$ is a minimal point and
in that case $x$ is called a {\em distal} point. A system is called a {\it distal system} when every point is
distal. A system $(X, T )$ with some distal point $x$ whose orbit $Orb(x)$ is dense in $X$ is called a
{\em point-distal system}.
For a weakly mixing system, the proximal cell $P[x]$ is ``big'', in the sense that
it is a residual subset of $X$ for all $x\in X$ \cite{AK03}. In fact, for mixing systems, proximal cells can be very complicated \cite{HSY04}.

\medskip

For a topological dynamical system $(X,T)$, a pair is said to be a {\em Li-Yorke pair} if it
is proximal but not asymptotic. $x\in X$ is {\it recurrent} if there is a subsequently $\{n_i\}$ of $\Z$
with $n_i\to \infty$ such that $T^{n_i}x\to x$.
A pair $(x,y) \in X^2 \setminus
\Delta_X $ is said to be a {\em strong Li-Yorke pair} if it is
proximal and is also a recurrent point of $X^2$, where $\Delta_X=\{(x,x):x\in X\}$.
A subset $A \subset X$ is called {\em scrambled}
({\em resp. strongly scrambled}) if every pair
of distinct points in $A$ is Li-Yorke (resp. strong Li-Yorke).
The system $(X,T)$
is said to be {\em Li-Yorke chaotic} ({\em resp. strong Li-Yorke} chaotic)
if it contains an uncountable scrambled (resp. strongly scrambled) subset.

\medskip

Let $(X,T)$ be a dynamical system. A subset $K$ of $X$ is {\em uniformly recurrent} if  for every $\ep>0$
there is an $n \in \N$ with $d(T^n x, x) < \ep$ for all $x$ in $K$.
$K$ is {\em recurrent}  if
every finite subset of $ K$ is  uniformly recurrent.
The subset $K$ is called {\em uniformly  proximal} if
for every $\ep>0$ there is
$n \in \N$ with ${\rm diam}\, T^n K < \ep$.
A subset $K$ of $X$ is called {\em proximal} if every finite subset of $K$ is
uniformly  proximal.

\medskip

A {\em homomorphism} $\pi : (X,T) \rightarrow (Y,S)$ is a
continuous surjective map from $X$ to $Y$ such that $S \circ \pi=\pi
\circ T$. In this case we say that
$(X,T)$ an {\em extension} of $(Y,S)$ and that $(Y,S)$ is  a {\em
factor} of $(X,T)$. An extension $\pi$ is determined by the
corresponding closed invariant equivalence relation $R_{\pi} = \{
(x_1,x_2): \pi x_1= \pi x_2 \} =(\pi \times \pi )^{-1} \Delta_Y
\subset  X \times X$.

\subsection{Li-Yorke sensitivity}\

\medskip

$(X,T)$ is {\em Li-Yorke sensitive}, briefly LYS or LYS$_\ep$, if there is an $\ep>0$ with the property that every $x\in X$ is a limit of points $y\in X$
such that the pair $(x,y)$ is proximal but not $\ep$-asymptotic, i.e., if
\begin{equation*}
  \liminf_{n\to\infty} \rho(T^nx,T^ny)=0, \text{and }\  \limsup_{n\to \infty}\rho(T^nx,T^ny)>\ep.
\end{equation*}
Each pair satisfying above condition is called an {\em $\ep$-Li-Yorke pair}. A set $S\subseteq X$ is called an
{\em $\ep$-scrambled set} if each pair with distinct elements is an $\ep$-Li-Yorke pair.
$(X,T)$ is {\em $\ep$-Li-Yorke chaotic}, briefly LYC$_\ep$ if it has an
uncountable $\ep$-scrambled subset, for some $\ep>0$.

\medskip

Li-Yorke Chaos \cite{LY75} and sensitivity \cite{Gu, GW93} are two basic notions to describe the complexity of a topological dynamics.
The notion of Li-Yorke sensitivity, combining the above two notions together, was introduced and studied by Akin and Kolyada \cite{AK03}. In \cite{AK03} the authors showed
that every nontrivial weak mixing system is LYS, and they stated five conjectures concerning LYS. Three of
them were disproved in \cite{CM06} and \cite{CM09}. In particular, it was
proved that a minimal LYS system needs not to have a nontrivial weak
mixing factor \cite{CM06}, and that a minimal system with a nontrivial LYS factor
needs not to be LYS \cite{CM09}.
The remaining two open problems are the following:

\begin{ques}\cite[Question 4.]{AK03}\label{Q1}
Is every minimal system with a nontrivial weak mixing factor
LYS?
\end{ques}

\begin{ques}\cite[Question 2.]{AK03}\label{Q2}
Does Li-Yorke sensitivity imply Li-Yorke chaos?
\end{ques}

In this paper we give an affirmative answer to Question \ref{Q1}. In fact, we show the following stronger result
(for the definition of PI, see the next subsection):

\medskip

\noindent {\bf Main Theorem.} \
{\em
Let $(X,T)$ be a minimal system. If
$(X,T)$ is not PI, then there is some $\ep>0$  such that for any $x\in X$ and any neighbourhood $U$ of $x$, there is a subset $S\subseteq U$ such that
\begin{enumerate}
  \item $S$ is uncountable and $x\in S \subseteq P[x]\cap U$;
  \item $S$ is $\ep$-scrambled;
  \item $S$ is proximal;
  \item $S$ is recurrent.
\end{enumerate}

In particular, $(X,T)$ is $\ep$-Li-Yorke chaotic, strongly Li-Yorke chaotic and Li-Yorke sensitive.
}

\medskip

It is known that a factor of a minimal PI system is PI (see Lemma \ref{PI-factor}). Using this fact it is easy to show that
every minimal system with a nontrivial weak mixing factor is not a PI system. So a minimal system with a non-trivial weakly
mixing factor is LYS, answering Question \ref{Q1} affirmatively. Moreover, by the Main Theorem,
we have that any non-PI system is (strongly) Li-Yorke chaotic, which was proved first in \cite{AGHSY}.
Note that here we offer a different approach.

\medskip

For more details related to Li-Yorke sensitivity, see \cite{AK03, CM06, CM09, CM16}.

\subsection{On the structure of minimal systems}\
\medskip

Our main tool to show the Main Theorem is the structure theorem of minimal systems. In this subsection
we state the structure theorem for minimal systems and give the definition of a PI system. For other undefined notions, see \cite{Au88, G76, V77}.

\medskip

We first recall definitions of extensions.
An extension $\pi :
(X,T) \rightarrow (Y,S)$ is called {\em proximal} if $R_{\pi}
\subset P(X,T)$. 
$\pi$ is an {\em equicontinuous} or {\em almost periodic}
extension if for every $\epsilon >0$ there is
$\delta >0$ such that $(x,y) \in R_{\pi}$ and $\rho (x,y)<\delta$ imply $\rho(T^nx,T^ny)<\epsilon$,
for every $n \in \Z$. In the metric case an equicontinuous extension is also called an
{\em isometric extension}.
The extension $\pi$ is a {\em weakly
mixing extension} when $(R_\pi, T\times T)$ as a subsystem of the product
system $(X\times X, T\times T)$ is transitive.
$\pi$ is called a {\em relatively incontractible (RIC) extension}\ if it is open and for every $n \ge 1$
the minimal points are dense in the relation
$$
R^n_\pi = \{(x_1,\dots,x_n) \in X^n : \pi(x_i)=\pi(x_j),\ \forall \ 1\le i
\le j \le n\}.
$$


\medskip
Note that $R_\pi^1=X$ and $R^2_\pi=R_\pi$.
We say that a minimal system $(X,T)$ is a
{\em strictly PI system} (PI means proximal-isometric) if there is an ordinal $\eta$
(which is countable when $X$ is metrizable)
and a family of systems
$\{(W_\iota,w_\iota)\}_{\iota\le\eta}$
such that

\begin{enumerate}
 \item $W_0$ is the trivial system,
  \item for every $\iota<\eta$ there exists a homomorphism
$\phi_\iota:W_{\iota+1}\to W_\iota$ which is
either proximal or equicontinuous
(isometric when $X$ is metrizable),
  \item for a
limit ordinal $\nu\le\eta$ the system $W_\nu$
is the inverse limit of the systems
$\{W_\iota\}_{\iota<\nu}$,
  \item $W_\eta=X$.
\end{enumerate}

We say that $(X,T)$ is a {\em PI-system} if there
exists a strictly PI system $\tilde X$ and a
proximal homomorphism $\theta:\tilde X\to X$.


\medskip

Finally we have the structure theorem for minimal systems
(see Ellis-Glasner-Shapiro \cite{EGS},
McMahon \cite{Mc76}, Veech \cite{V77}, and
Glasner \cite{G76}).

\begin{thm}[Structure theorem for minimal systems]\label{structure}
Let $(X,T)$ be a minimal system. Then we have the following diagram:

$$
\xymatrix{
  X          & X_\infty \ar[d]^{\pi_\infty} \ar[l]^{\theta} \\
                & Y_\infty             }
$$
\medskip

\noindent where $X_\infty$ is a proximal extension of $X$ and a RIC
weakly mixing extension of the strictly PI-system $Y_\infty$.

The homomorphism $\pi_\infty$ is an isomorphism (so that
$X_\infty=Y_\infty$) if and only if  $X$ is a PI-system.
\end{thm}

\begin{rem}\label{rem-wm}
If $(X,T)$ is a weakly mixing minimal system, then in the structure theorem, $X=X_\infty$ and $Y_\infty$ is trivial. Hence if a minimal system is both PI and weakly mixing, then it is trivial.
\end{rem}

\section{Proof of Main Theorem}

In this section we will give the proofs of the main results of the paper.
To this aim, we need some basic results from the theory of minimal flows.

\medskip

First recall some basic notions related to the  Ellis semigroup.
Given a system $(X,T)$ its {\em enveloping semigroup} or {\em Ellis
semigroup} $E(X,T)$ is defined as the closure of the set $\{T^n: n\in
\Z\}$ in $X^X$ (with its compact, usually non-metrizable, pointwise
convergence topology).
Let $(X,T),(Y,S)$ be systems and $\pi: X\rightarrow Y$ be an
extension. Then there is a unique continuous semigroup homomorphism
$\pi^* : E(X,T)\rightarrow E(Y,S)$ such that
$\pi(px)=\pi^*(p)\pi(x)$ for all $x\in X,p\in E(X,T)$. When there
is no confusion, we usually regard that the enveloping semigroup of $X$
acts on $Y$: $p\pi(x)=\pi(px)$ for $x\in X$ and $p\in E(X,T)$.

For a semigroup the element $u$ with $u^2=u$ is called an {\em
idempotent}. The well known Ellis-Numakura theorem states that for any enveloping
semigroup $E$ the set $J(E)$ of idempotents of $E$ is not empty.
An idempotent $u \in J(E)$ is {\em minimal} if $v \in
J(E)$ and $vu=v$ implies $uv=u$.
A point $x\in X$ is minimal if and only if $ux=x$ for some minimal idempotent $u\in E(X,T)$.

\medskip
For $n\in \N$, let
$T^{(n)}=T\times T\times \ldots \times T$ ($n$ times). For a subset $A\subseteq X^n$, let
$$Orb (A, T^{(n)})=\left\{\big(T^{(n)}\big)^kA: k\in \Z\right\}.$$

The following theorem is crucial to the proof of the Main Theorem.
Note that we assume that $X$ and $Y$ are metrizable.

\begin{thm}\cite[Lemma B.2 and Theorem B.3]{SY} \label{thm-SY}
Let $(X,T)$ and $(Y,T)$ be minimal systems and let $\pi: X
\rightarrow Y$ be a RIC weakly mixing extension. Let $y\in Y$
with $uy=y$, where $u$ is a minimal idempotent. Then for all $n\ge 2$, any
nonempty open subset $U$ of $\overline{u\pi^{-1}(y)}$ and any
transitive point $x'=(x_1', \cdots, x_{n-1}')\in R^{n-1}_\pi$ with
$\pi(x_j')=y, j=1,\cdots, n-1$, one has that
$$\overline{Orb(\{x'\}\times U, T^{(n)})}=R^n_\pi.$$

Moreover, for each transitive point $x'=(x_1', \cdots, x_{n-1}')\in R^{n-1}_\pi$ with $\pi(x_j')=y, j=1,\cdots, n-1$, there is some
residual subset $D$ of $\overline{u\pi^{-1}(y)}$ such that if
$x''\in D$, then \break $\overline{Orb((x',x''),T^{(n)})}=R^n_{\pi}$.
\end{thm}

\begin{rem} We have the following
\begin{enumerate}
  \item To prove Theorem \ref{thm-SY} one needs the so-called Ellis trick by Glasner in \cite{G76}.
  We refer to \cite{Gl05} for more discussions about weakly mixing extensions.
  \item In Theorem \ref{thm-SY}, when $n=2$, $R^{n-1}_\pi=R^1_\pi=X$, i.e. for each $x\in X$ with $\pi(x)=y$ there is some residual subset $D$ of $\overline{u\pi^{-1}(y)}$ such that if
$x'\in D$, $\overline{Orb((x,x'),T\times T)}=R_{\pi}$. This result is Theorems 14.27 and 14.28 in \cite{Au88}.
\end{enumerate}
\end{rem}

\begin{cor}\label{cor-perfect}
Let $(X,T)$ and $(Y,T)$ be minimal  systems and let $\pi: X
\rightarrow Y$ be a RIC weakly mixing extension. Let $y\in Y$
with $uy=y$, where $u$ is a minimal idempotent. If $\pi$ is non-trivial, then  $\overline{u\pi^{-1}(y)}$ is perfect, and hence it has the cardinality of the continuum.
\end{cor}

\begin{proof}
If $\overline{u\pi^{-1}(y)}$ is not perfect, then there is some $x\in \overline{u\pi^{-1}(y)}$ such that $\{x\}$ is relatively open in $\overline{u\pi^{-1}(y)}$. Take $n=2$ and $U=\{x\}$ in Theorem \ref{thm-SY}, then we have
$$\D_X=\overline{Orb(\{x\}\times \{x\}, T\times T)}=R_\pi.$$
Thus $\pi$ is trivial, a contradiction! The proof is completed.
\end{proof}

The following proposition will be used in the proof of the Main Theorem.

\begin{prop}\cite[Corollary 7.4]{EGS}\label{PI-factor} 
A factor of a PI flow is PI.
\end{prop}

Now we are ready  to give the proof of the Main Theorem.

\begin{proof}[Proof of the Main Theorem]
By the structure theorem for minimal systems we have the following diagram:
$$
\xymatrix{
  X          & X_\infty \ar[d]^{\pi_\infty} \ar[l]^{\theta} \\
                & Y_\infty             }
$$
\medskip

\noindent where $\theta$ is a proximal extension, $\pi_\infty$ is a non-trivial weakly mixing RIC extension and $Y_\infty$ is a strictly PI-system. Let
$$d=\max \{\rho(\theta( x_1),\theta( x_2)): (x_1,x_2)\in R_{\pi_\infty}\}.$$
Then $d>0$, since if $d=0$ then  $R_{\pi_\infty}\subseteq R_\theta$, which implies that $X$ is a factor of $Y_\infty$,
a contradiction by Proposition \ref{PI-factor}.  Put $\ep= \frac12 d>0$.

\medskip

Fix $x\in X$ and choose $x'_1\in \theta^{-1}(x)$. Let $ux'_1=x'_1$ for some minimal idempotent $u$ and $U'=\theta^{-1}(U)$, where
$U$ is a nonempty open subset of $\overline{u\pi^{-1}(y)}$ with $y=\pi_\infty(x'_1)$. Then $U'$ is a neighborhood of $x'_1$.
We are going to construct increasing subsets $X_\alpha\subset  \overline{u\pi_\infty^{-1}(y)}\cap U'$,
$\alpha<\Omega$ such that each non-empty finite subset $F$ of $X_\alpha$ is a transitive point of $R_{\pi_\infty}^{|F|}$, where $\Omega$ is the first
uncountable ordinal number.

Put $X_1=\{x_1'\}$. By Theorem \ref{thm-SY} and Corollary \ref{cor-perfect}, there is some
$x_2'\in \overline{u\pi_\infty^{-1}(y)}\cap U'$ such that $(x_1',x'_2)$ is a
transitive point of $R^2_{\pi_\infty}$. Let $X_2=\{x_1', x_2'\}$.

Now assume that $\alpha<\Omega$ is an ordinal number and $X_\beta$ has been constructed for any $\beta<\alpha$ such that
each non-empty finite subset $F$ of $X_\beta$ is a transitive point of $R_{\pi_\infty}^{|F|}$.
If $\alpha$ is a limit ordinal number then we put $X_\alpha=\cup_{\beta<\alpha}X_\beta$. It is clear that each finite non-empty subset $F$ of $X_\alpha$
is a transitive point of $R_{\pi_\infty}^{|F|}$.
Assume now $\alpha$ is not a limit ordinal number. For each finite non-empty subset $F$ of $X_{\alpha-1}$, by Theorem \ref{thm-SY}
and Corollary \ref{cor-perfect}, the set $Z_F$ of points $z\in \overline{u\pi_\infty^{-1}(y)}$ such that
$(F,z)$ is transitive in $R^{|F|+1}_{\pi_\infty}$ is residual in $\overline{u\pi_\infty^{-1}(y)}$. Hence
$Z=\cap_{F} Z_F$ is also residual in $\overline{u\pi_\infty^{-1}(y)}$, where $F$ runs over all finite non-empty subsets of $X_{\alpha-1}$.
Choose $x_\alpha'\in (Z\setminus X_{\alpha-1})\cap U'$ and put $X_\alpha=X_{\alpha-1}\cup \{x_\alpha'\}$.
It is clear that $X_\alpha\subset  \overline{u\pi_\infty^{-1}(y)}\cap U'$ and each finite non-empty subset $F$ of
$X_\alpha$ is a transitive point of $R_{\pi_\infty}^{|F|}$.


\medskip

Let $S=\cup_{\a<\Omega} \theta (X_{\a})$. Then $x=\theta(x_1')\in S$ and $S\subseteq U$. First we show that all points in $S$ are distinct, and hence $S$ is uncountable.
Assume the contrary that not all points in $S$ are distinct. This means that there are some $\a<\b<\Omega$ such that
$\theta (x_\a')=\theta(x_\b').$
In particular,
\begin{equation*}
  (x_\a',x_\b')\in R_{\theta}.
\end{equation*}
As $(x_\a',x_\b')$ is a transitive point of $(R_{\pi_\infty}, T\times T)$, and it follows that
$$R_{\pi_\infty}= \overline{Orb((x_\a',x_\b'),T\times T)}\subseteq R_\theta.$$ Hence $X$ is a factor of $Y_\infty$,
and $X$ is a PI flow by Proposition \ref{PI-factor}, a contradiction.

\medskip

Now we show that $S$ is proximal and recurrent. Let $n\in \N$ and $\a_1<\a_2<\ldots < \a_n<\Omega$.
Let ${\bf x'}=(x'_{\a_1},x'_{\a_2},\ldots, x'_{\a_n})$ and ${\bf x}=(x_{\a_1},x_{\a_2},\ldots, x_{\a_n})$, where $x_{\a_i}=\theta(x_{\a_i}'), 1\le i\le n$.
Since the diagonal of ${X_\infty}^n$ is contained in $R^n_{\pi_\infty}$ and ${\bf x'}$ is the transitive point of
$R^n_{\pi_{\infty}}$, there are a sequence $\{n_i\}\subseteq \Z$ and some point $z'\in X_\infty$
such that $(T^{n_i}x_{\a_1}',\ldots, T^{n_i}x_{\a_n}')\to (z',\ldots,z')$ as $n_i\to \infty$.
Since $\theta$ is continuous,
$$(T^{n_i}x_{\a_1},\ldots, T^{n_i}x_{\a_n})=(\theta(T^{n_i}x'_{\a_1}),\ldots,
\theta(T^{n_i}x'_{\a_n}))\to (\theta(z'),\ldots, \theta(z')),\ n_i\to\infty.$$
Thus $\{x_{\a_1},x_{\a_2},\ldots, x_{\a_n} \}$ is uniformly proximal.

Since ${\bf x'}$ is the transitive point of
$R^n_{\pi_{\infty}}$, ${\bf x'}$  is recurrent in the product system $(X^n_{\infty}, T^{(n)})$ and hence its $\theta\times \ldots,\times\theta$-image ${\bf x}$ is also recurrent. So $S$ is recurrent.

\medskip

It is left to show that for all $\a<\b<\Omega$, $(x_\a,x_\b)=(\theta(x_\a'), \theta(x_\b'))$ is LYS$_\ep$.
By the definition of $d$, there is some pair $(p,q)\in R_{\pi_\infty}$ such that $\rho(\theta( p),\theta( q))>\frac 12 d=\ep$. Since $(x'_{\a},x'_{\b})$ is the transitive point of
$R_{\pi_{\infty}}$, there is a sequence $\{m_i\}\subseteq \Z$ such that $(T^{m_i}x_\a', T^{m_i}x_\b')\to (p,q)$ as $m_i\to \infty$. As $\theta$ is continuous,
$$\lim_i\rho(T^{m_i}x_\a, T^{m_i}x_\b)=\lim_i\rho(\theta(T^{m_i}x'_\a), \theta(T^{m_i}x'_\b))= \rho(\theta(p),\theta(q))>\ep.$$
Hence $(x_\a,x_\b)$ is LYS$_\ep$.
The proof is completed.
\end{proof}

\begin{cor}\label{cor-AK}
Any minimal system with nontrivial weakly mixing factor is Li-Yorke sensitive.
\end{cor}

\begin{proof}
Let $(X,T)$ be a  minimal system with nontrivial weakly mixing factor $(Y,T)$.
By the Main Theorem, it remains to show that $(X,T)$ is not PI.
If not, then as factor of $(X,T)$, $(Y,T)$ is PI by Proposition \ref{PI-factor}. Hence $(Y,T)$ is both weakly mixing and PI, which means that $(Y,T)$ is trivial (Remark \ref{rem-wm}). A contradiction!
\end{proof}


\begin{rem}\label{rem-invertible}
In this paper, we consider $T: X\rightarrow X$ as a homeomorphism for simplicity. When $T$ is a continuous surjective map, one can use the natural extension to get the corresponding results. Recall that for a system $(X,T)$, let $(\widetilde{X},\widetilde{T})$
be the {\em natural extension} of $(X,T)$, i.e.
$\widetilde{X}=\{(x_1,x_2,\ldots)\in \prod_{i=1}^\infty X:
T(x_{i+1})=x_i,\ i\in\N\}$ (as the subspace of the product space)
and $\widetilde{T}(x_1,x_2,\ldots)=(T(x_1), x_1,x_2,\ldots)$. It is clear
that $\widetilde{T}:\widetilde{X}\rightarrow \widetilde{X}$ is a homeomorphism and $p_1: \widetilde{X}\rightarrow X$ is an asymptotic extension, where
$p_1$ is the projection to the first coordinate. When $(X,T)$ is minimal, $(\widetilde{X},\widetilde{T})$ is minimal and $p_1$ is almost one-to-one \cite[Corollary 5.18]{AGHSY}.

\end{rem}

\section{A question}

Finally we state a question related to Question \ref{Q2}.

\begin{ques}\label{Q3}
Is a non point-distal minimal system Li-Yorke chaotic?
\end{ques}

Since each LYS minimal system is not point-distal, an affirmative answer to Question \ref{Q3} will also give an affirmative answer to Question \ref{Q2} for the minimal case.

\medskip
\noindent{\bf Acknowledgments:} We would like to thank Dr. Jian Li for very useful suggestions. 


\end{document}